\documentclass{amsart}
\usepackage[latin1]{inputenc}
\usepackage[T1]{fontenc}
\usepackage{lmodern}
\usepackage[english]{babel}
\usepackage{microtype}

\usepackage{amsmath,amssymb,amsfonts,amsthm}
\usepackage{mathtools,accents}
\usepackage{mathrsfs}
\usepackage{aliascnt}
\usepackage{braket}
\usepackage{bm}
\usepackage[citecolor=blue,colorlinks]{hyperref}
\addto\extrasenglish{}

\usepackage{enumitem}
\usepackage{enumerate}
\usepackage{xcolor}

\usepackage{aliascnt}

\makeatletter
\def\newaliasedtheorem#1[#2]#3{
  \newaliascnt{#1@alt}{#2}
  \newtheorem{#1}[#1@alt]{#3}
  \expandafter\newcommand\csname #1@altname\endcsname{#3}
}
\makeatother

\theoremstyle{plain}
\newtheorem{theorem}{Theorem}[section]
\newaliasedtheorem{lemma}[theorem]{Lemma}
\newaliasedtheorem{proposition}[theorem]{Proposition}
\newaliasedtheorem{prop}[theorem]{Proposition}
\newaliasedtheorem{claim}[theorem]{Claim}
\newaliasedtheorem{corollary}[theorem]{Corollary}

\theoremstyle{definition}
\newaliasedtheorem{definition}[theorem]{Definition}
\newaliasedtheorem{example}[theorem]{Example}

\theoremstyle{remark}
\newaliasedtheorem{remark}[theorem]{Remark}
\numberwithin{equation}{section}

\def\R{{\mathbb R}}
\def\C{{\mathbb C}}
\def\Q{\mathbf Q}
\def\P{\mathbf P}
\def\d{{\partial}}
\def\a{{\alpha}}
\def\i{{\infty}}

\DeclareMathOperator{\diver}{div}

\begin{document}

\title{THE ELECTROSTATIC LIMIT FOR THE 3D ZAKHAROV SYSTEM}

\author{Paolo Antonelli}
\address{GSSI, Gran Sasso Science Institute, Viale F. Crispi 7,  67100 L'Aquila, Italy}
\email{paolo.antonelli@gssi.infn.it}

\author{Luigi Forcella}
\address{Scuola Normale Superiore, Piazza dei Cavalieri, 7, 56126 Pisa, Italy}
\email{luigi.forcella@sns.it}



\begin{abstract}
We consider the vectorial Zakharov system describing Langmuir waves in a weakly magnetized plasma. In its original derivation \cite{Z} the evolution for the electric field envelope is governed by a Schr\"odinger type equation with a singular parameter which is usually large in physical applications. Motivated by this, we study the rigorous limit as this parameter goes to infinity. By using some Strichartz type estimates to control separately the fast and slow dynamics in the problem, we show that the evolution of the electric field envelope is asymptotically constrained onto the space of irrotational vector fields.
\end{abstract}

\maketitle
\section{Introduction.}\label{introduction}

In this paper we consider the vectorial Zakharov system \cite{Z} describing Langmuir waves in a weakly magnetized plasma.
After a suitable rescaling of the variables it reads \cite{SS}
\begin{equation}\label{ZG}
\left\{ \begin{array}{ll}
i\partial_tu-\alpha\nabla\times\nabla\times u+\nabla(\diver u)=nu\\
\frac{1}{c_s^2} \partial_{tt}n-\Delta n=\Delta|u|^2
\end{array} \right.,
\end{equation}

\noindent subject to initial conditions
\begin{equation*}
u(0)=u_0, \quad n(0)=n_0, \quad\d_tn(0)=n_1.
\end{equation*}
Here $u:\R\times\R^3\to\C^3$ describes the slowly varying envelope of the highly oscillating electric field, whereas $n:\R\times\R^3\to\R$ is the ion density fluctuation.
The rescaled constants in \eqref{ZG} are $\alpha=\frac{c^2}{3v_e^2}$, $c$ being the speed of light and $v_e=\sqrt{\frac{T_e}{m_e}}$ the electron thermal velocity, while $c_s$ is proportional to the ion acoustic speed.
In many physical situations the parameter $\alpha$ is relatively large, see for example table 1, p. 47 in \cite{TtH}, hence hereafter we will only consider $\alpha\geq1$. In the large $\alpha$ regime, the electric field is almost irrotational and in the electrostatic limit $\alpha\to\infty$ the dynamics is asymptotically described by
\begin{equation}\label{ZL}
\left\{ \begin{array}{ll}
i\partial_tu+\Delta u=\Q(nu) \\
\frac{1}{c_s^2} \partial_{tt}n-\Delta n=\Delta|u|^2
\end{array} \right.,
\end{equation}
where $\Q=-(-\Delta)^{-1}\nabla\diver$ is the Helmholtz projection operator onto irrotational vector fields. 
By further simplifying \eqref{ZG} it is possible to consider the so called scalar Zakharov system
\begin{equation}\label{Z}
\left\{ \begin{array}{ll}
i\partial_tu+\Delta u=nu \\
\frac{1}{c_s^2} \partial_{tt}n-\Delta n=\Delta|u|^2
\end{array} \right.,
\end{equation}
which retains the main features of \eqref{ZL}. In the subsonic limit $c_s\to\infty$ we find the cubic focusing nonlinear Schr\"odinger equation
\begin{equation*}
i\d_tu+\Delta u+|u|^2u=0.
\end{equation*}

The Cauchy problem for the Zakharov system has been extensively studied in the mathematical literature. For the local and global well-posedness, see \cite{SS2, OT1, OT2, KPV, BC} and the recent results concerning low regularity solutions \cite{GTV, BH}. In \cite{M} formation of blow-up solutions is studied by means of virial identities, see also \cite{GM} where self-similar solutions are constructed in two space dimensions. The subsonic limit $c_s\to\infty$ for \eqref{Z} is investigated in \cite{SW}. Furthermore, some related singular limits are also studied in \cite{MN}, considering the Klein-Gordon-Zakharov system. Here in this paper we do not consider such limits, hence without loss of generalities we can set $c_s=1$.

The aim of our research is to rigorously study the electrostatic limit for the vectorial Zakharov equation, namely we show that mild solutions to \eqref{ZG} converge towards solutions to \eqref{ZL} as $\alpha\to\infty$. 

As we will see below, we will investigate this limit by exploiting two auxiliary systems associated to \eqref{ZG}, \eqref{ZL}, namely systems \eqref{ZM2} and \eqref{eq4.5} below. 
Those are obtained by considering $v=\d_tu$ as a new variable and by studying the Cauchy problem for the auxiliary system describing the dynamics for $(v, n)$ and a state equation for $u$ (see \autoref{sect:LWP} for more details).
This approach, already introduced in \cite{OT1, OT2} to study local and global well-posedness for the Zakharov system \eqref{Z}, overcomes the problem generated by the loss of derivatives on the term $|u|^2$ in the wave equation, but in our context it introduces a new difficulty. Indeed the initial data $v(0)$ is not uniformly bounded for $\alpha\geq1$, see also the beginning of \autoref{sect:conv} below for a more detailed mathematical discussion.

For this reason we will need to consider a family of well-prepared initial data; more precisely we will take a set $u_0^\alpha$ of initial states for the Schr\"odinger part in \eqref{ZG} which converges to an irrotational initial datum for \eqref{ZL}.

We consider initial data $(u_0^\alpha, n_0^\alpha, n_1^\alpha)\in H^2(\R^3)\times H^1(\R^3)\times L^2(\R^3)=:\mathcal H_2$ for \eqref{ZG}, converging in the same space to a set of initial data 
$(u_0^\infty, n_0^\infty, n_1^\infty)\in \mathcal{H}_2$, with $u_0^\infty$ an irrotational vector field, and we show the convergence in the space
\begin{equation*}
\begin{aligned}
\mathcal X_T:=\big\{(u, n)\;:\;&u\in L^q(0, T;W^{2, r}(\R^3)),\;\forall\;(q, r)\;\textrm{admissible pair},\\
&n\in L^\infty(0, T;H^1(\R^3))\cap W^{1, \infty}(0, T;L^2(\R^3))\big\}.
\end{aligned}
\end{equation*}

For a more detailed discussion about notations and the spaces considered in this paper we refer the reader to \autoref{sect:prel}.

Before stating our main result we first recall the local well-posedness result in $\mathcal H_2$ for system \eqref{ZL}.
\begin{theorem}[\cite{OT1}]\label{lwpOT1}
Let $(u_0, n_0, n_1)\in\mathcal H_2$, then there exist a maximal time $0<T_{max}\leq\infty$ and a unique solution $(u, n)$ to \eqref{ZL} such that 
$u\in\mathcal C([0, T_{max}); H^2)\cap\mathcal C^1([0, T_{max});L^2)$, $n\in\mathcal C([0, T_{max});H^1)\cap\mathcal C^1([0, T_{max});L^2)$.
Furthermore the solution depends continuously on the initial data and the standard blow-up alternative holds true: either $T_{max}=\infty$ and the solution is global or $T_{max}<\infty$ and we have
\begin{equation*}
\lim_{t\rightarrow T_{max}}\|(u, n, \d_tn)(t)\|_{\mathcal H_2}=\infty.
\end{equation*}
\end{theorem}
Analogously we are going to prove the same local well-posedness result for system \eqref{ZG}. 
Moreover, despite the fact that the initial datum for \eqref{ZM2}  is not uniformly bounded for $\alpha\geq1$ (see the discussion at the beginning of \autoref{sect:LWP}), we can anyway infer some a priori bounds in $\alpha$ for the solution $(u^\alpha, n^\alpha)$ to \eqref{ZG}.
\begin{theorem}\label{thm:lwp}
Let $(u^\alpha_0, n^\alpha_0, n^\alpha_1)\in\mathcal H_2$, then there exist a maximal time $T^\alpha_{max}>0$ and a unique solution $(u^\alpha, n^\alpha)$ to \eqref{ZG} such that

\begin{itemize}

\item $u^\alpha\in\mathcal C([0, T_{max}^\alpha); H^2)\cap\mathcal C^1([0, T_{max}^\alpha);L^2)$,
\item $n^\alpha\in\mathcal C([0, T_{max}^\alpha);H^1)\cap\mathcal C^1([0, T_{max}^\alpha);L^2)$.
\end{itemize}
Furthermore the existence times $T^\alpha_{max}$ are uniformly bounded from below, $0<T^\ast\leq T^\alpha_{max}$ for any $\alpha\geq1,$ and we have
\begin{equation*}
\|(u^\alpha, n^\alpha, \d_tn^\alpha)\|_{L^\infty(0, T;\mathcal H_2)}+\|\d_tu^\alpha\|_{L^2(0, T;L^6)}\leq C(T, \|u_0^\alpha, n_0^\alpha, n_1^\alpha\|_{\mathcal H_2}),
\end{equation*}
for any $0<T<T^\alpha_{max}$, where the constant above does not depend on $\alpha\geq1$.
\end{theorem}

Our main result in this paper is the following one.

\begin{theorem}\label{thm:main}
Let $(u_0^\alpha, n_0^\alpha, n_1^\alpha)\in \mathcal{H}_2$ and let $(u^\alpha, n^\alpha)$ be the maximal solution to \eqref{ZG} defined on the time interval $[0, T_{max}^\alpha)$. Let us assume that
\begin{equation*}
\lim_{\alpha\to\infty}\|(u_0^\alpha, n_0^\alpha, n_1^\alpha)-(u_0^\infty, n_0^\infty, n_1^\infty)\|_{\mathcal{H}_2}=0,
\end{equation*}
for some $(u_0^\infty, n_0^\infty, n_1^\infty)\in \mathcal{H}_2$ such that $u_0^\infty=\Q u_0^\infty$, and let $(u^\infty, n^\infty)$ be the maximal solutions to \eqref{ZL} in the interval $[0, T_{max}^\infty)$ with such initial data. Then
\begin{equation*}
\liminf_{\alpha\to\infty}T_{max}^\alpha\geq T^\infty_{max}
\end{equation*}
and we have the following convergence
\begin{equation*}
\lim_{\alpha\to\infty}\|(u^\alpha, n^\alpha)-(u^\infty, n^\infty)\|_{\mathcal X_T}=0,
\end{equation*}
for any $0<T<T_{max}^\infty$.
\end{theorem}
The paper is structured as follows. In \autoref{sect:prel} we fix some notations and give some preliminary results which will be used in the analysis of the problem below. In \autoref{sect:LWP} we show the local well-posedness of system \eqref{ZG} in the space $\mathcal{H}_2$. Finally in \autoref{sect:conv} we investigate the electrostatic limit and prove the main theorem.
\subsection*{Acknowledgements} 
This paper and its project originated after many useful discussions with Prof. Pierangelo Marcati, during second author's M. Sc. thesis work. We would like to thank P. Marcati for valuable suggestions.

\section{Preliminary results and tools.}\label{sect:prel}
In this section we introduce notations and some preliminary results which will be useful in the analysis below. The Fourier transform of a function $f$ is defined by
\begin{equation*}
\mathcal F(f)(\xi)=\hat f(\xi)=\int_{\R^3}e^{-2\pi ix\cdot\xi}f(x)\,dx,
\end{equation*}
with its inverse 
\begin{equation*}
f(x)=\int_{\R^3}e^{2\pi ix\cdot\xi}\hat f(\xi)\,d\xi.
\end{equation*}
Given an interval $I\subset\mathbb{R},$ we denote by $L^q(I;L^r)$ the Bochner space equipped with the norm defined by
$$\|f\|_{L^q(I;L^r)}=\bigg(\int_{I}\|f(s)\|_{L^r(\mathbb{R}^3)}^q\,ds\bigg)^{1/q},$$
where $f=f(s,x).$ When no confusion is possible, we write $L^q_tL^r_x=L^q(I;L^r(\R^3))$. Given two Banach spaces $X, Y$, we denote $\|f\|_{X\cap Y}:=\max\{\|f\|_X, \|f\|_Y\}$ for $f\in X\cap Y$. 
With $W^{k,p}$ we denote the standard Sobolev spaces and for $p=2$ we write $H^k=W^{k, 2}$.
$A\lesssim B$ means that there exists a universal constant $C$ such that $A\leq CB$ and in general in a chain of inequalities the constant may change from one line to the other.\newline
As already said in the Introduction, given a vector field $F$, we denote by $\Q F=-(-\Delta)^{-1}\nabla\diver F$ its projection into irrotational fields, moreover $\P=1-\Q$ is its orthogonal projection operator onto solenoidal fields. 
Let us just recall that $\nabla\times F$ is the standard curl operator on $\R^3$.

\noindent The space of initial data is denoted by $\mathcal{H}_2:=H^2(\R^3)\times H^1(\R^3)\times L^2(\R^3)$. A pair of Lebesgue exponents is called \emph{Schr\"odinger admissible} (or simply admissible) if $2\leq q\leq\infty$, $2\leq r\leq 6$ and they are related through
\begin{equation*}
\frac1q=\frac32\left(\frac12-\frac1r\right).
\end{equation*}
Given a time interval $I\subset\R$ we denote the Strichartz space $S^0$(I) to be the closure of the Schwartz space with the norm
\begin{equation*}
\|u\|_{S^0(I)}:=\sup_{(q, r)}\|u\|_{L^q(I;L^r(\R^3))},
\end{equation*}
where the $\sup$ is taken over all admissible pairs; furthermore we write $$S^2(I)=\{u\in S^0(I)\;:\;\nabla^2u\in S^0(I)\}.$$ 

\noindent We define moreover the space 
\begin{equation*}
\mathcal W^1(I)=\{n\,:\,n\in L^\infty(I;H^1)\cap W^{1, \infty}(I;L^2)\}
\end{equation*}
endowed with the norm 
\begin{equation*}
\|n\|_{\mathcal W^1(I)}=\|n\|_{L^{\i}(I;H^1)}+\|\d_tn(t)\|_{L^{\i}(I;L^2)}.
\end{equation*}
The space of solutions we consider in this paper is given by
\begin{equation*}
\mathcal X_T=\{(u, n)\;:\;u\in S^2([0, T]),\;n\in \mathcal W^1([0,T])\}.
\end{equation*}
We will also use the following notation:
\begin{equation*}
\begin{aligned}
\mathcal C([0, T); \mathcal{H}_2)=\big\{(u, n)\;:\;&u\in\mathcal C([0, T);H^2)\cap\mathcal C^1([0, T);L^2),\\
&n\in\mathcal C([0, T);H^1)\cap\mathcal C^1([0, T);L^2)\big\}.
\end{aligned}
\end{equation*}
Here in this paper we only consider positive times, however the same results are valid also for negative times.

We now introduce some basic preliminary results which will be useful later in the analysis.

First of all we consider the linear propagator related to \eqref{ZG}, namely
\begin{equation}\label{f1.3}
i\d_tu=\alpha\nabla\times\nabla\times u-\nabla\diver u.
\end{equation}
\begin{lemma}
Let $u$ solve \eqref{f1.3} with initial datum $u(0)=u_0$, then
\begin{equation}\label{eq:prop}
u(t)=U_Z(t)u_0=\left[U(\alpha t)\P+U(t)\Q\right]u_0,
\end{equation}
where $U(t)=e^{it\Delta}$ is the Schr\"odinger evolution operator.
\end{lemma}
\begin{proof}
By taking the Fourier transform \eqref{f1.3} we have
\begin{equation*}
\begin{aligned}
i\d_t\hat u&=-\alpha\xi\times\xi\times\hat u+\xi(\xi\cdot\hat u)\\
&=|\xi|^2\left(\alpha\hat\P(\xi)+\hat\Q(\xi)\right)\hat u(\xi),
\end{aligned}
\end{equation*}
where $\hat\P(\xi), \hat\Q(\xi)$ are two $(3\times3)-$matrices defined by $\hat\Q(\xi)=\frac{\xi\otimes\xi}{|\xi|^2}$, $\hat\P(\xi)=\bold1-\hat\Q(\xi)$ where $\bold1$ is the identity matrix . Hence we may write
\begin{equation*}
\hat u(t)=e^{-i\alpha t|\xi|^2\hat\P(\xi)-it|\xi|^2\hat\Q(\xi)}\hat u_0(\xi).
\end{equation*}
It is straightforward to see that $\hat\Q(\xi)$ is a projection matrix, $0\leq\hat\Q(\xi)\leq1$, $\hat\Q(\xi)=\hat\Q^2(\xi)$, hence $\hat\P(\xi)$ is its orthogonal projection. Consequently we have
\begin{equation*}
\begin{aligned}
\hat u(t)&=e^{-i\alpha t|\xi|^2\hat\P(\xi)}e^{-it|\xi|^2\hat\Q(\xi)}\hat u_0(\xi)\\
&=\left(e^{-i\alpha t|\xi|^2}\hat\P(\xi)+\hat\Q(\xi)\right)\left(e^{-it|\xi|^2}\hat\Q(\xi)+\hat\P(\xi)\right)\hat u_0(\xi)\\
&=\left(e^{-i\alpha t|\xi|^2}\hat\P(\xi)+e^{-it|\xi|^2}\hat\Q(\xi)\right)\hat u_0(\xi).
\end{aligned}
\end{equation*}
By taking the inverse Fourier transform we find \eqref{eq:prop}.
\end{proof}
By the dispersive estimates for the standard Schr\"odinger evolution operator (see for example \cite{C},\cite{GV}, \cite{Y}), we have
\begin{equation}\label{eq:disp_pq}
\begin{aligned}
\|U(t)\Q f\|_{L^p}&\lesssim|t|^{-3\left(\frac12-\frac1p\right)}\|\Q f\|_{L^{p'}}\\
\|U(\alpha t)\P f\|_{L^p}&\lesssim|\alpha t|^{-3\left(\frac12-\frac1p\right)}\|\P f\|_{L^{p'}},
\end{aligned}
\end{equation}
for any $2\leq p\leq\infty$, $t\neq0$. These two estimates together give
\begin{equation*}
\|U_Z(t)f\|_{L^p_x}\lesssim|t|^{-3\left(\frac12-\frac1p\right)}\|f\|_{L^{p'}_x},
\end{equation*}
for $2\leq p<\infty$. Let us notice that the dispersive estimate for $p=\infty$ does not hold for $U_Z(t)$ anymore because the projection operators $\Q, \P$ are not bounded from $L^1$  into itself. 
Nevertheless by using the dispersive estimates in \eqref{eq:disp_pq} and the result in \cite{KT} we infer the whole set of Strichartz estimates for the irrotational and solenoidal part, separately. By summing them up we thus find the Strichartz estimates for the propagator in \eqref{eq:prop}. 

\begin{lemma}\label{lemma:strich}
Let $(q, r)$, $(\gamma, \rho)$ be two arbitrary admissible pairs and let $\alpha\geq1$, then we have
\begin{align}
\|U(\alpha t)\P f\|_{L^q_t(I;L^r_x)}&\leq C \alpha^{-\frac2q}\|f\|_{L^2_x},\label{e5.20}\\
\bigg\|\int_{0}^{t}U(\alpha(t-s))\P F(s)\,ds\bigg\|_{L^q_t(I;L^r_x)}&\leq C\alpha^{-\left(\frac{1}{q}+\frac{1}{\gamma}\right)}\|F\|_{L^{\gamma^\prime}_t(I;L^{\rho^\prime}_x)}.\notag
\end{align}
and
\begin{equation*}
\begin{aligned}
\|U(t)\Q f\|_{L^q_t(I;L^r_x)}&\leq C \|f\|_{L^2_x},\\
\bigg\|\int_{0}^{t}U(t-s)\Q F(s)\,ds\bigg\|_{L^q_t(I;L^r_x)}&\leq C\|F\|_{L^{\gamma^\prime}_t(I;L^{\rho^\prime}_x)},
\end{aligned}
\end{equation*}
Consequently we also have 
\begin{align}\label{eq2.10}
\|U_Z(t)g\|_{L^q_t(I;L^r_x)}&\leq C\|f\|_{L^2_x},\\\label{eq2.11}
\bigg\|\int_{0}^{t}U_Z(t-s)F(s)\,ds\bigg\|_{L^q_t(I;L^r_x)}&\leq C\|F\|_{L^{\gamma^\prime}_t(I;L^{\rho^\prime}_x)}.
\end{align}
\end{lemma}

\begin{remark}
The following remarks are in order.
\begin{itemize}[leftmargin=+.2in]
\item From the estimates in the Lemma above it is already straightforward that, at least in the linear evolution, we can separate the fast and slow dynamics and that the fast one is asymptotically vanishing. This is somehow similar to what happens with rapidly varying dispersion management, see for example \cite{ASS}.
\item Let us notice that the constants in \eqref{eq2.10} and \eqref{eq2.11} are uniformly bounded for $\alpha\geq1$. This is straightforward but it is a necessary remark to infer that the existence time in the local well-posedness section is uniformly bounded from below for any $\alpha\geq1$.
\end{itemize}
\end{remark}

\section{Local existence theory.}\label{sect:LWP}
In this Section we study the local well-posedness of \eqref{ZG} in the space $\mathcal{H}_2$.
We are going to perform a fixed point argument in order to find a unique local solution in the time interval $[0, T]$, for some $0<T<\infty$. By standard arguments it is then possible to extend the solution up to a maximal time $T_{max}$ for which the blow-up alternative holds.
However, due to the loss of derivatives on the term $|u|^2$, we cannot proceed in a straightforward way, thus we follow the approach in \cite{OT1} where the authors use an auxiliary system to overcome this difficulty. 
More precisely, let us define $v:=\d_tu$, then by differentiating the Schr\"odinger equation in \eqref{ZG}  with respect to time, we write the following system
\begin{equation}\label{ZM2}
\left\{ \begin{array}{lll}
i\partial_tv-\alpha\nabla\times\nabla\times v+\nabla\diver v=nv+\partial_tn u \\
\partial_{tt}n-\Delta n=\Delta|u|^2 \\
iv-\alpha\nabla\times\nabla\times u+\nabla\diver u=nu
\end{array} \right..
\end{equation}
Differently from \cite{OT1}, here we encounter a further difficulty. Indeed we have that the initial datum for $v$ is given by
\begin{equation}\label{eq:v0id}
v(0)=-i\alpha\nabla\times\nabla\times u_0+i\nabla\diver u_0-in_0u_0,
\end{equation}
which in general is not uniformly bounded in $L^2$ for $\alpha\geq1$. 
Hence the standard fixed point argument applied to the integral formulation of \eqref{ZM2} would give a local solution on a time interval $[0, T^\alpha]$, where $T^\alpha$ goes to zero as $\alpha$ goes to infinity.
For this reason we  introduce the alternative variable
\begin{equation}\label{eq:v_tilde}
\tilde v(t):=v(t)-U(\alpha t)\P(i\alpha\Delta u_0),
\end{equation}
for which we prove that the existence time $T^\alpha$ is uniformly bounded from below for $\alpha\geq1$.
The main result of this Section concerns the local well-posedness for \eqref{ZM2}.
\begin{proposition}\label{prop:small_times}
Let $(u_0, n_0, n_1)\in \mathcal{H}_2$ be such that
\begin{equation*}
M:=\|(u_0, n_0, n_1)\|_{\mathcal{H}_2}.
\end{equation*}
Then, for any $\alpha\geq1$ there exists $\tau=\tau(M)$ and a unique local solution $(u, n)\in\mathcal C([0, \tau]; \mathcal{H}_2)$ to \eqref{ZG} such that
\begin{equation*}
\sup_{[0, \tau]}\|(u, n, \d_tn)(t)\|_{\mathcal{H}_2}\leq2M
\end{equation*}
and
\begin{equation*}
\|v\|_{L^2_tL^6_x}\leq CM,
\end{equation*}
where $C$ does not depend on $\alpha\geq1$.
\end{proposition}
By standard arguments we then extend the local solution in \autoref{prop:small_times} to a maximal existence interval where the standard blow-up alternative holds true.
\begin{theorem}\label{frthm1}
Let $(u_0, n_0, n_1)\in \mathcal{H}_2$, then for any $\alpha\geq1$ there exists a unique maximal solution $(u^\alpha, v^\alpha, n^\alpha)$ to \eqref{ZM2} with initial data $(u_0, v(0), n_0, n_1)$, $v(0)$ given by \eqref{eq:v0id}, on the maximal existence interval $I_\alpha:=[0, T_{max}^\alpha)$, for some $T_{max}^\alpha>0$. The solution satisfies the following regularity properties:
\begin{itemize}
\item $u^{\alpha}\in\,\mathcal{C}(I_\alpha;H^2), \;u^\alpha\in S^2([0, T]),\;\forall\;0<T<T^\alpha_{max}$,
\item $v^{\alpha}\in\,\mathcal C(I_\alpha;L^2),\,\;v^\alpha\in S^0([0, T]),\;\forall\;0<T<T^\alpha_{max}$,
\item $n^{\alpha}\in\,\mathcal{C}(I_\alpha;H^1)\cap\mathcal C^1(I_\alpha; L^2)$.
\end{itemize}

\noindent Moreover, the following blow-up alternative holds true: $T^\alpha_{max}<\infty$ if and only if
\begin{equation*}
\lim_{t\to T^\alpha}\|(u^\alpha, n^\alpha)(t)\|_{\mathcal{H}_2}=\infty.
\end{equation*}
Finally, the map $\mathcal H_2\to\mathcal C([0, T_{max});\mathcal H_2)$ associating any initial datum to its solution is a continuous operator.
\end{theorem} 
\begin{remark}
The blow-up alternative above also implies in particular that the family of maximal existence times $T^\alpha$ is strictly bounded from below by a positive constant, i.e. there exists a $T^\ast>0$ such that $T^\ast\leq T^\alpha$ for any $\a\geq1$.
\end{remark}
\autoref{thm:lwp} yields in a straightforward way from \autoref{frthm1} above.
\begin{proof}[Proof of \autoref{thm:lwp}]
Let $(u^\alpha, v^\alpha, n^\alpha)$ be the solution to \eqref{ZM2} constructed in \autoref{frthm1}, then to prove the \autoref{thm:lwp} we only need to show that we identify $\d_tu^\alpha=v^\alpha$ in the distribution sense.
Let us differentiate with respect to $t$ the equation $$(1-\alpha\Delta\P-\Delta\Q)u=iv-(n-1)\bigg(u_0+\int_{0}^tv(s)\,ds\bigg)$$ obtaining 
\begin{equation}\label{e1.32}
(1-\alpha\Delta\P-\Delta\Q)\partial_tu=i\partial_tv-(n-1)v-\partial_tn\bigg(u_0+\int_{0}^t v(s)\,ds\bigg),
\end{equation}
this equation holding in $H^{-2},$ while the first equation of \eqref{ZM2} gives us 
\begin{equation*}
(1-\alpha\Delta\P-\Delta\Q)v=i\partial_tv-(n-1)v-\partial_tn\bigg(u_0+\int_{0}^tv(s)\,ds\bigg).
\end{equation*}
Also the equation above is satisfied in $H^{-2}$ and therefore in the same distributional sense we have 
\begin{equation*}
\partial_tu=v.
\end{equation*}
Moreover from \eqref{e1.32} we get 
\begin{equation}\notag
\partial_tu=(1-\alpha\Delta\P-\Delta\Q)^{-1}\bigg(i\partial_tv-(n-1)v-\partial_tn\bigg(u_0+\int_{0}^tv(s)\,ds\bigg)\bigg)\in\mathcal{C}(I;L^2)
\end{equation}
therefore $u\in\mathcal{C}^1(I;L^2).$ It is straightforward that $u^\a(0,x)=u_0$ and so the proof is complete. 
\end{proof}

\begin{proof}[Proof of \autoref{frthm1}]
As discussed above, we are going to prove the result by means of a fixed point argument. Let us define the function
\begin{equation*}
\tilde v(t):=v(t)-U(\alpha t)\P(i\alpha\Delta u_0).
\end{equation*}
We look at the integral formulation for \eqref{ZM2}, namely
\begin{equation}\label{eq:duh_v}
\begin{aligned}
v(t)=U_Z(t)v(0)-i\int_0^tU_Z(t-s)\left(nv+\d_tnu\right)(s)\,ds\\
\end{aligned}
\end{equation}
\begin{equation*}
n(t)=\cos(t|\nabla|)n_0+\frac{\sin(t|\nabla|)}{|\nabla|}n_1+\int_0^t\frac{\sin((t-s)|\nabla|)}{|\nabla|}\Delta|u|^2\,ds,
\end{equation*}
with $u$ determined by the following elliptic equation
\begin{equation*}
-\alpha\nabla\times\nabla\times u+\nabla\diver u=n\left(u_0+\int_0^tv(s)\,ds\right)-iv,
\end{equation*}
and $v(0)$ is given by \eqref{eq:v0id}. This implies that $\tilde v$ must satisfy the following integral equation
\begin{equation*}
\begin{aligned}
\tilde v(t)&=U(\alpha t)\P(-in_0u_0)+U(t)\Q(i\Delta u_0-in_0u_0)\\
&\phantom{\quad}-i\int_0^tU_Z(t-s)\left(\tilde vn+nU(\alpha\cdot)\P(i\alpha\Delta u_0)+\d_tnu\right)(s)\,ds.
\end{aligned}
\end{equation*}
Let us consider the space
\begin{equation*}
\begin{aligned}
X=\big\{(\tilde v,n):\,&\tilde v\in S^2([0, T]), n\in \mathcal W^1([0,T]),\\
&\|\tilde v\|_{S^2(I)}\leq M, \|n\|_{\mathcal W^1(I)}\leq M\big\},
\end{aligned}
\end{equation*}
endowed with the norm
\begin{equation*}
\|(\tilde v, n)\|_X:=\|\tilde v\|_{S^2(I)}+\|n\|_{\mathcal W^1(I)}.
\end{equation*}
Here $0<T\leq1, M>0$ will be chosen subsequently and $I:=[0,T]$.
From the third equation in \eqref{ZM2} and the definition of $\tilde v$ we have
\begin{equation}\label{eq:102}
\begin{aligned}
-\alpha\nabla\times\nabla\times u+\nabla\diver u&= -i \tilde{v}-iU(\a t)(i\a\Delta\P u_0)\\
&\phantom{\quad\,}-in\bigg(u_0+\int_0^t\tilde{v}(s)+U(\a s)(i\a\Delta\P u_0)\,ds\bigg),
\end{aligned}
\end{equation}
thus it is straightforward to see that given $n, \tilde v$, then $u$ is uniquely determined. Furthermore, by applying the projection operators $\P, \Q$, respectively, to \eqref{eq:102} we obtain
\begin{equation*}
\a\Delta\P u=-i\P[\tilde{v}+U(\a t)\P(i\a\Delta u_0)]+\P\bigg[n \bigg(u_0+\int_0^t\tilde{v}(s)+U(\a s)\P(i\a\Delta u_0)\,ds \bigg)\bigg]
\end{equation*}
and 
\begin{equation*}
\Delta \Q u=-i\Q\tilde{v}+\Q \bigg [n \bigg(u_0+\int_0^t\tilde{v}(s)+U(\a s)\P(i\a\Delta u_0)\,ds \bigg) \bigg].
\end{equation*}
We now estimate the irrotational and solenoidal parts of $\Delta u$ separately. 
Let us start with $\Q\Delta u:$ by H\"older inequality and Sobolev embedding we obtain
\begin{equation*}
\begin{aligned}
\|\Delta\Q u\|_{L^\infty_tL^2_x}&\lesssim\|\tilde{v}\|_{L^\infty_tL^2_x}+\|n\|_{L^\infty_tH^1_x}\|u_0\|_{H^2}+T^{1/2}\|n\|_{L^\infty_tH^1_x}\|\tilde v\|_{L^{2}_tL^6_x}\\
&\phantom{\quad\,}+T^{1/2}\|n\|_{L^\infty_tH^1_x}\|U(\alpha t)\P(i\alpha\Delta u_0)\|_{L^2_tL^6_x}.
\end{aligned}
\end{equation*}
To estimate the last term, we use the Strichartz estimate in \eqref{e5.20}; let us notice that by choosing the admissible exponents $(q, r)=(2, 6)$ we obtain a factor $\alpha^{-1}$ in the estimate, which balances the term $\alpha$ appearing above. We thus have
\begin{equation*}
\|\Delta\Q u\|_{L^\infty_tL^2_x}\lesssim(\|u_0\|_{H^2}+1)M+M^2.
\end{equation*}
By similar calculations, we also obtain an estimate for $\P\Delta u$,
\begin{equation*}
\|\P\Delta u\|_{L^\infty_tL^2_x}\lesssim \|u_0\|_{H^2}^2+\|u_0\|_{H^2}M+M^2.
\end{equation*}
We then sum up the contributions given by the irrotational and solenoidal parts to get
\begin{equation}\label{eq:3.15}
\|u\|_{L^\infty_tH^2_x}\lesssim \|u_0\|_{H^2}^2+\|u_0\|_{H^2}M+M^2\leq C(\|u_0\|_{H^2})\big(1+M^2\big).
\end{equation}

\noindent Similar calculations also give
\begin{equation*}
\begin{aligned}
\|u-u'\|_{L^\infty(I;H^2)}&\lesssim\|\tilde v-\tilde v'\|_{L^{\infty}_tL^2_x}+\|n-n'\|_{L^\infty_tH^1_x}\\
&\phantom{\quad\,}+M(\|n-n'\|_{L^\infty_tH^1_x}+\|\tilde v-\tilde v'\|_{L^{2}_tL^6_x})\\
&\leq C(1+M)\|(\tilde v, n)-(\tilde v', n')\|_X.
\end{aligned}
\end{equation*}

Given $(\tilde v, n)\in X$ we define the map $\Phi:X\to X$, $\Phi(\tilde v, n)=(\Phi_S, \Phi_W)(\tilde v, n)$ by
\begin{align}
\Phi_S&=U(\alpha t)\P(-in_0u_0)+U(t)\Q(i\Delta u_0-in_0u_0)\label{eq:fixs}\\
&\phantom{\quad\,}-i\int_0^tU(\alpha(t-s))\P(\tilde vn+nU(\alpha\cdot)\P(i\alpha\Delta u_0)+\d_tnu)(s)\,ds\notag\\
&\phantom{\quad\,}-i\int_0^tU(t-s)\Q\left(n\tilde v+nU(\alpha\cdot)(i\alpha\Delta u_0)+\d_tnu\right)(s)\,ds\notag\\
\Phi_W&=\cos(t|\nabla|)n_0+\frac{\sin(t|\nabla|)}{|\nabla|}n_1+\int_0^t\frac{\sin((t-s)|\nabla|)}{|\nabla|}\Delta|u|^2(s)\,ds,\label{eq:fixw}
\end{align}
where $u$ in the formulas above is given by \eqref{eq:102} and its $L^\infty_tH^2_x$ norm is bounded in \eqref{eq:3.15}. Let us first prove that, by choosing $T$ and $M$ properly, $\Phi$ maps $X$ into itself. 

Let us first analyze the Schr\"odinger part \eqref{eq:fixs}, by the Strichartz estimates in \autoref{lemma:strich}, H\"older inequality and Sobolev embedding we have
\begin{equation*}
\|U(\alpha t)\P(-in_0u_0)+U(t)\Q(i\Delta u_0-in_0u_0)\|_{L^{q}L^r}\lesssim\|u_0\|_{H^2}+\|n_0\|_{H^1}\|u_0\|_{H^2}\\
\end{equation*}
We treat the inhomogeneous part similarly,
\begin{equation*}
\begin{aligned}
\bigg\|\int_0^tU_Z(t-s)\left(n\tilde v+nU(\alpha s)(i\alpha\P \Delta u_0)\right)(s)\,ds\bigg\|_{L^q_tL^r_x}&\lesssim\|n\tilde v+nU(\alpha \cdot)(i\alpha\Delta\P u_0)\|_{L^1_tL^2_x}\\
\lesssim T^{1/2}\|n\|_{L^\infty_tH^1_x}(\|\tilde v\|_{L^2_tL^6_x}+\|U(\alpha t)\P(i\alpha\Delta u_0)\|_{L^2_tL^6_x})&\lesssim T^{1/2}M(M+\|u_0\|_{H^2}).
\end{aligned}
\end{equation*}
where in the last inequality we again used \eqref{e5.20} with $(2, 6)$ as admissible pair. Similarly,
\begin{equation*}
\begin{aligned}
\bigg\|\int_0^tU_Z(t-s)\left(\d_tnu\right)(s)\,ds\bigg\|_{L^q_tL^r_x}&\lesssim T\|\d_tn\|_{L^\infty_tL^2_x}\|u\|_{L^\infty_tH^2_x}\\
&\lesssim C(\|u_0\|_{H^2})TM\big(1+M^2\big),
\end{aligned}
\end{equation*}
where in the last line we use the bound \eqref{eq:3.15}.
Collecting these estimates we get
\begin{equation}\label{eq:201}
\|\Phi_S(\tilde v, n)\|_{L^q_tL^r_x}\leq C(\|u_0\|_{H^2}, \|n_0\|_{L^2})+CT^{1/2}M(1+M).
\end{equation}

\noindent For the wave component we use formula \eqref{eq:fixw} and H\"older inequality to obtain
\begin{equation*}
\begin{aligned}
\|\Phi_W(v, n)\|_{\mathcal W^1(I)}&\leq C(1+T)\|n_0\|_{H^1}+\|n_1\|_{L^2}+\|\Delta|u|^2\|_{L^1_tL^2_x}\\
&\leq C\left(\|n_0\|_{H^1}+\|n_1\|_{L^2}\right)+T\|u\|^2_{L^\infty_tH^2_x},
\end{aligned}
\end{equation*}
where we used the fact that $H^2(\R^3)$ is an algebra. From \eqref{eq:3.15} we infer
\begin{equation}\label{eq:202}
\|\Phi_W(v, n)\|_{\mathcal W^1(I)}\leq C(\|n_0\|_{H^1}, \|n_1\|_{L^2})+T\big(M+M^4\big).
\end{equation}
The bounds \eqref{eq:201} and \eqref{eq:202} together yield
\begin{equation*}
\|\Phi(\tilde v, n)\|_{X}\leq C(\|(u_0, n_0, n_1)\|_{\mathcal H_2})+CT^{1/2}M(1+M^3).
\end{equation*}
Let us choose $M$ such that 
\begin{equation*}
\frac{M}{2}=C(\|(u_0, n_0, n_1)\|_{\mathcal H_2})
\end{equation*}
and $T$ such that
\begin{equation*}
CT^{1/2}(1+M^3)<\frac12,
\end{equation*}
we then obtain $\|\Phi(\tilde v, n)\|_X\leq M$.
Hence $\Phi$ maps $X$ into itself. It thus remains to prove that $\Phi$ is a contraction.
Arguing similarly to what we did before we obtain 
\begin{equation*}
\begin{aligned}
\|\Phi_S(\tilde{v}, n)-\Phi_S(\tilde{v}', n')\|_{L^{q}_tL^r_x}&\leq C T^{1/2}(1+M)\|(\tilde v,n)-(\tilde v',n')\|_{L^{q}_tL^r_x}\\
\|\Phi_W(\tilde{v}, n)-\Phi_W(\tilde{v}', n')\|_{\mathcal W^1(I)}&\leq C T\big(1+M^3\big)\|(\tilde v,n)-(\tilde v',n')\|_{\mathcal W^1(I)}.
\end{aligned}
\end{equation*}

\noindent By possibly choosing a smaller $T>0$ such that $C T^{1/2}(1+M^3)<1$ 
then we see that $\Phi:X\to X$ is a contraction and consequently there exists a unique $(\tilde v, n)\in X$ which is a fixed point for $X$. Let us notice that the time $T$ depends only on $M$, hence 
$T=T(\|(u_0, n_0, n_1)\|_{\mathcal H_2})$. Furthermore from the definition of $\tilde v$ it follows that $(u, v, n)$ is a solution to \eqref{ZM2}, where $v=\tilde v+U(\alpha t)\P(i\alpha\Delta u_0)$. 
From \eqref{eq:3.15} we also see that the $L^\infty_tH^2_x$ norm of $u$ is uniformly bounded in $\alpha$.

Finally, from standard arguments we extend the solution on a maximal time interval, on which the standard blow-up alternative holds true and we can also infer the continuous dependence on the initial data.
\end{proof}

\section{Convergence of solutions.}\label{sect:conv}
Given the well-posedness results of the previous Section, we are now ready to study the electrostatic limit for the vectorial Zakharov system \eqref{ZG}. In order to understand the effective dynamics we consider the system \eqref{ZG} in its integral formulation, by splitting the Schr\"odinger linear propagator in its fast and slow dynamics, i.e. $U_Z(t)=U(\alpha t)\P+U(t)\Q$. In particular for $u^\alpha$ we have
\begin{equation*}
u^\alpha(t)=U(\alpha t)\P u_0+U(t)\Q u_0-i\int_0^tU( \alpha(t-s))\P(nu)(s)\,ds-i\int_0^tU(t-s)\Q(nu)(s)\,ds.
\end{equation*}
Due to fast oscillations, we expect that the terms of the form $U(\alpha t)f$ go weakly to zero as $\alpha\to0$. This fact can be quantitatively seen by using the Strichartz estimates in \eqref{e5.20}. However, while for the third term we can choose $(\gamma, \rho)$ in a suitable way such that it converges to zero in every Strichartz space, by the unitarity of $U(\alpha t)$ we see that $\|U(\alpha t)\P u_0\|_{L^\infty_tL^2_x}$ cannot converge to zero, while $\|U(\alpha t)\P u_0\|_{L^q_tL^r_x}\to0$ for any admissible pair $(q, r)\neq(\infty, 2)$.

This is indeed due to the presence of an initial layer for the electrostatic limit for \eqref{ZG} when dealing with ``ill-prepared'' initial data. In general, for arbitrary initial data, the right convergence should be given by
\begin{equation*}
\tilde u^\alpha(t):=u^\alpha(t)-U(\alpha t)\P u_0\to u^\infty
\end{equation*}
in all Strichartz spaces, where $u^\infty$ is the solution to \eqref{ZL}. Let us notice that $\tilde u^\alpha$ is related to the auxiliary variable $\tilde v^\alpha$ defined in \eqref{eq:v_tilde} and used to prove the local well-posedness results in \autoref{sect:LWP}, since we have $\tilde v^\alpha=\d_t\tilde u^\alpha$.

Our strategy to prove the electrostatic limit goes through studying the convergence of $(v^\alpha, n^\alpha, u^\alpha)$, studied in the previous Section, towards solutions to 
\begin{equation}\label{eq4.5}
\left\{ \begin{array}{lll}
i\partial_tv^{\infty}+\Delta v^{\infty}=\Q(n^{\infty}v^{\infty}+\partial_tn^{\infty} u^{\infty}) \\
\partial_{tt}n^{\infty}-\Delta n^{\infty}=\Delta|u^{\infty}|^2 \\
iv^{\infty}+\Delta u^{\infty}=\Q(n^{\infty}u^{\infty}),
\end{array} \right.
\end{equation}  
which is the auxiliary system associated to \eqref{ZL}. Again, we exploit such auxiliary formulations in order to overcome the difficulty generated by the loss of derivatives on the terms $|u^\alpha|^2$ and $|u^\infty|^2$. 

Unfortunately our strategy is not suitable to study the limit in the presence of an initial layer. Indeed for ill-prepared data we should consider $\tilde u^\alpha$ and consequently $\tilde v^\alpha$ defined in 
\eqref{eq:v_tilde} for the auxiliary system. This means that when studying the auxiliary variable $v^\alpha$ the initial layer itself becomes singular. For this reason here we restrict ourselves to study the limit with well-prepared data. More specifically, we consider $(u_0^\alpha, n_0^\alpha, n_1^\alpha)\in \mathcal{H}_2$ such that
\begin{equation}\label{eq:conv_id}
\|(u_0^\alpha, n_0^\alpha, n_1^\alpha)-(u_0^\infty, n_0^\infty, n_1^\infty)\|_{\mathcal{H}_2}\to0
\end{equation}
for some $(u_0^\infty, n_0^\infty, n_1^\infty)\in \mathcal{H}_2$ and 
\begin{equation}\label{eq:wp}
\|\P u_0^\alpha\|_{H^2}\to0.
\end{equation}
This clearly implies that the initial datum for the limit equation \eqref{ZL} is irrotational, i.e. $\P u_0^\infty=0$.

\begin{remark}
In view of the above discussion, it is reasonable to think about studying the initial layer by considering the Cauchy problem for the Zakharov system in low regularity spaces, by exploiting recent results in \cite{BC, GTV, BH}. However this goes beyond the scope of our paper and it could be the subject of some future investigations.
\end{remark}


To prove the convergence result stated in \autoref{thm:main} we will study the convergence from \eqref{ZM2} to \eqref{eq4.5}. The main result of this Section is the following.
\begin{theorem}\label{mainth}
Let $\alpha\geq1$ and let $(u_0^\alpha, n_0^\alpha, n_1^\alpha)$, $(u_0^\infty, n_0^\infty, n_1^\infty)\in \mathcal{H}_2$ be initial data such that \eqref{eq:conv_id} and \eqref{eq:wp} hold true. Let 
$(u^\alpha, v^\alpha, n^\alpha)$ be the maximal solution to \eqref{ZM2} with Cauchy data $(u_0^\alpha, n_0^\alpha, n_1^\alpha)$ given by \autoref{frthm1} 
and analogously let $(u^\infty, v^\infty, n^\infty)$ be the maximal solution to \eqref{eq4.5} in the interval $[0, T^\infty_{max})$ accordingly to \autoref{lwpOT1}. Then for any $0<T<T^\infty_{max}$ we have
\begin{equation*}
\lim_{\alpha\to\infty}\|(u^\alpha, v^\alpha, n^\alpha)-(u^\infty, v^\infty, n^\infty)\|_{L^\infty(0, T;\mathcal H_2)}=0.
\end{equation*}
\end{theorem}
The proof of the Theorem above is divided in two main steps. First of all we prove in \autoref{lemma4.2} that, as long as the $\mathcal{H}_2$ norm of $(u^\alpha(T), n^\alpha(T),\d_tn^\a(T))$ is bounded, then the convergence holds true in $[0, T]$. The second one consists in proving that the $\mathcal{H}_2$ bound on $(u^\alpha(T), n^\alpha(T), \d_tn^\alpha(T))$ holds true for any $0<T<T^\infty_{max}$. A similar strategy of proof is already exploited in the literature to study the asymptotic behavior of time oscillating nonlinearities, see for example \cite{CS} where the authors consider a time oscillating nonlinearity or \cite{AW} where in a system of two nonlinear Schr\"odinger equations a rapidly varying linear coupling term is averaging out the effect of nonlinearities. We also mention \cite{CPS} where a similar strategy is also used to study a time oscillating critical Korteweg-de Vries equation.

\begin{lemma}\label{lemma4.2}
Let $(u^\alpha, v^\alpha, n^\alpha)$, $(u^\infty, v^\infty, n^\infty)$ be defined as in the statement of \autoref{mainth} and let us assume that for some $0<T_1<T^\infty_{max}$ we have
\begin{equation*}
\sup_{\alpha\geq1}\|(u^\alpha, n^\alpha, \d_tn^\alpha)\|_{L^\infty(0, T_1;\mathcal H_2)}<\infty.
\end{equation*}
It follows that
\begin{equation*}
\lim_{\alpha\to\infty}\left(\|u^\alpha-u^\infty\|_{L^\infty_tH^2_x}+\|v^\alpha-v^\infty\|_{L^2_tL^6_x}+\|n^\alpha-n^\infty\|_{\mathcal W^1}\right)=0,
\end{equation*}
where all the norms are taken in the space-time slab $[0, T_1]\times\R^3$. In particular we have
\begin{equation*}
\lim_{\alpha\to\infty}\|(u^\alpha, n^\alpha, \d_tn^\alpha)-(u^\infty, n^\infty, \d_tn^\infty)\|_{L^\infty(0, T_1;\mathcal H_2)}=0.
\end{equation*}
\end{lemma}

We assume for the moment that \autoref{lemma4.2} holds true, then we first show how this implies \autoref{mainth}.

\begin{proof}[Proof of \autoref{mainth}.]
Let $0<T<T_{max}^\infty$ be fixed and let us define
\begin{equation*}
N:=2\|(u^\infty, n^\infty, \d_tn^\infty)\|_{L^\infty(0, T;\mathcal H_2)}.
\end{equation*}
From the local well-posedness theory, see \autoref{prop:small_times}, there exists $\tau=\tau(N)$ such that the solution $(u^\alpha, n^\alpha, \d_n^\alpha)$ to \eqref{ZM2} exists on $[0, \tau]$ and we have
\begin{equation*}
\|(u^\alpha, n^\alpha, \d_tn^\alpha)\|_{L^\infty(0, T_1;\mathcal H_2)}<\infty.
\end{equation*}
We observe that, because of what we said before, the choice $T_1=\tau$ is always possible. By the \autoref{lemma4.2} we infer that
\begin{equation*}
\lim_{\alpha\to\infty}\|(u^\alpha, n^\alpha, \d_tn^\alpha)-(u^\infty, n^\infty, \d_tn^\infty)\|_{L^\infty(0, T_1;\mathcal H_2)}=0.
\end{equation*}
On the other hand by the definition of $N$ we have that, for $\alpha\geq1$ large enough,
\begin{equation*}
\begin{aligned}
\|(u^\alpha, n^\alpha, \d_tn^\alpha)(T_1)\|_{\mathcal{H}_2}&\leq\|(u^\alpha, n^\alpha, \d_tn^\alpha)(T_1)-(u^\infty, n^\infty, \d_tn^\infty)(T_1)\|_{\mathcal{H}_2}\\
&\phantom{\quad\,}+\|(u^\infty, n^\infty, \d_tn^\infty)(T_1)\|_{\mathcal{H}_2}\leq N.
\end{aligned}
\end{equation*}
Consequently we can apply \autoref{prop:small_times} to infer that $(u^\alpha, n^\alpha)$ exists on a larger time interval $[0, T_1+\tau]$, provided $T_1+\tau\leq T$, and again
\begin{equation*}
\|(u^\alpha, n^\alpha, \d_tn^\alpha)\|_{L^\infty(0, T_1+\tau;\mathcal H_2)}\leq 2N.
\end{equation*}
We can repeat the argument iteratively on the whole interval $[0, T]$ to infer
\begin{equation*}
\|(u^\alpha, n^\alpha, \d_tn^\alpha)\|_{L^\infty(0, T;\mathcal H_2)}\leq 2N.
\end{equation*}
By using \autoref{lemma4.2} this proves the Theorem.
\end{proof}

It only remains now to prove \autoref{lemma4.2}.

\begin{proof}[Proof of \autoref{lemma4.2}]
Let us fix
\begin{equation*}
M:=\sup_{\alpha}\sup_{[0, T_1]}\|(u^\alpha, n^\alpha, \d_tn^\alpha)(t)\|_{\mathcal{H}_2}.
\end{equation*}
By using the integral formulation for \eqref{ZM2} and \eqref{eq4.5} we have
\begin{equation*}
\begin{aligned}
v^{\alpha}(t)-v^{\infty}(t)&=U(\alpha t)\P(\alpha\Delta u^{\alpha}_0-iu^{\alpha}_0n^{\alpha}_0)+U(t)\Q(v^{\alpha}_0-v^{\infty}_0)\\
&\phantom{\quad\,}-i\int_{0}^t U(\alpha(t-s))[\P  (\partial_t(n^{\alpha}u^{\alpha}))](s)\,ds\\
&\phantom{\quad\,}-i\int_{0}^t U(t-s)[\Q  (\partial_t(n^{\alpha}u^{\alpha})-\partial_t(n^{\infty}u^{\infty}))](s)\,ds.
\end{aligned}
\end{equation*}
Now we use the Strichartz estimates in \autoref{lemma:strich} to get
\begin{equation*}
\begin{aligned}
\|v^\alpha-v^\infty\|_{L^2_tL^6_x}&\lesssim\|\P u_0^\alpha\|_{H^2}+\alpha^{-1}\|n_0^\alpha\|_{H^1}\|u_0^\alpha\|_{H^2}+\|v^\alpha_0-v^\infty_0\|_{L^2}\\
&\phantom{\quad\,}+\alpha^{-1/2}\|n^\alpha v^\alpha+\d_tn^\alpha u^\alpha\|_{L^1_tL^2_x}\\
&\phantom{\quad\,}+\|n^\alpha v^\alpha-n^\infty v^\infty\|_{L^1_tL^2_x}+\|\d_tn^\alpha u^\alpha-\d_tn^\infty u^\infty\|_{L^1_tL^2_x}.
\end{aligned}
\end{equation*}
It is straightforward to check that, by H\"older inequality and Sobolev embedding,
\begin{equation*}
\begin{aligned}
\|n^\alpha v^\alpha+\d_tn^\alpha u^\alpha\|_{L^1_tL^2_x}&\leq C(T, M),\\
\|n^\alpha v^\alpha-n^\infty v^\infty\|_{L^1_tL^2_x}&\lesssim T^{1/2}(\|n^\alpha-n^\infty\|_{L^\infty_tH^1_x}+\|v^\alpha-v^\infty\|_{L^2_tL^6_x}),\\
\|\d_tn^\alpha u^\alpha-\d_tn^\infty u^\infty\|_{L^1_tL^2_x}&\lesssim T\left(\|\partial_tn^\alpha-\d_tn^\infty\|_{L^\infty_tL^2_x}+\|u^\alpha-u^\infty\|_{L^\infty_tH^2_x}\right).
\end{aligned}
\end{equation*}
By putting all the estimates together we obtain
\begin{equation*}
\begin{aligned}
\|v^\alpha-v^\infty\|_{L^2_tL^6_x}&\lesssim\|\P u_0^\alpha\|_{H^2}+\alpha^{-1}\|n_0^\alpha\|_{H^1}\|u_0^\alpha\|_{H^2}\|u_0^\alpha-u_0^\infty\|_{H^2}+\alpha^{-1/2}+\|n_0^\alpha-n_0^\infty\|_{H^1}\\
&\phantom{\quad\,}+T^{1/2}(\|u^\alpha-u^\infty\|_{L^\infty_tH^2_x}+\|v^\alpha-v^\infty\|_{L^2_tL^6_x}+\|n^\alpha-n^\infty\|_{\mathcal W^1}).
\end{aligned}
\end{equation*}
To estimate the wave part in \eqref{ZM2} and \eqref{eq4.5}, we write
\begin{equation*}
\begin{aligned}
n^{\a}-n^{\i}&=\cos(t|\nabla|)(n^{\a}_0-n^{\i}_0)-\frac{\sin(t|\nabla|)}{|\nabla|}(n^{\a}_1-n^{\i}_1)\\
&\phantom{\quad\,}+\int_0^t\frac{\sin((t-s)|\nabla|)}{|\nabla|}\Delta(|u^{\a}|^2-|u^{\i}|^2)\,ds,
\end{aligned}
\end{equation*}
whence, by using again that $H^2(\R^3)$ is an algebra, 
\begin{equation*}
\|n^\alpha-n^\infty\|_{\mathcal W^1}\lesssim\|n_0^\alpha-n_0^\infty\|_{H^1}+\|n_1^\alpha-n_1^\infty\|_{L^2}+T\|u^\alpha-u^\infty\|_{L^\infty_tH^2_x}.
\end{equation*}
The estimate for the difference $u^\alpha-u^\infty$ is more delicate. From the third equations in \eqref{ZM2} and \eqref{eq4.5} we have
\begin{equation*}
-\alpha\nabla\times\nabla\times u^\alpha+\nabla\diver(u^\alpha-u^\infty)=i(v^\alpha-v^\infty)-n^\alpha u^\alpha+\Q(n^\infty u^\infty).
\end{equation*}
Again, here we estimate separately the irrotational and solenoidal parts of the difference. For the solenoidal part we obtain
\begin{equation*}
\alpha\|\P\Delta u^\alpha\|_{L^\infty_tL^2_x}\lesssim\|v^\alpha\|_{L^\infty_tL^2_x}+\|n^\alpha u^\alpha\|_{L^\infty_tL^2_x}.
\end{equation*}
To estimate the $L^\infty_tL^2_x$ norm of $v^\alpha$ on the right hand side we use \eqref{eq:duh_v} and Strichartz estimates to infer
\begin{equation*}
\|v^\alpha\|_{L^\infty_tL^2_x}\lesssim\alpha\|\P u_0^\alpha\|_{H^2}\|u_0^\alpha\|_{H^2}\|n_0^\alpha\|_{H^1}+1.
\end{equation*}
Hence
\begin{equation*}
\alpha\|\P\Delta u^\alpha\|_{L^\infty_tL^2_x}\lesssim\alpha\|\P u^\a_0\|_{H^2}+\|u_0^\alpha\|_{H^2}\|n_0\|_{H^1}+1.
\end{equation*}
For the irrotational part
\begin{equation}\label{eq:u_diff}
\|\Q\Delta(u^\alpha-u^\infty)\|_{L^\infty_tL^2_x}\lesssim\|\Q(v^\alpha-v^\infty)\|_{L^\infty_tL^2_x}+\|n^\alpha-u^\alpha-n^\infty u^\infty\|_{L^\infty_tL^2_x}.
\end{equation}
By using \eqref{eq:duh_v}, the analogue integral formulation for $v^\infty$ and by applying the Helmholtz projection operator $\Q$ to their difference we have
that the first term on the right hand side is bounded by
\begin{equation*}
\begin{aligned}
\|\Q(v^\alpha-v^\infty)\|_{L^\infty_tL^2_x}&\lesssim\|u_0^\alpha-u_0^\infty\|_{H^2}+\|n_0^\alpha-n_0^\infty\|_{H^1}\\
&\phantom{\quad\,}+T^{1/2}\left(\|u^\alpha-u^\infty\|_{L^\infty_tH^2_x}+\|v^\alpha-v^\infty\|_{L^2_tL^6_x}+\|n^\alpha-n^\infty\|_{\mathcal W^1}\right).
\end{aligned}
\end{equation*}
The second term on the right hand side of \eqref{eq:u_diff} is estimated by
\begin{equation*}
\begin{aligned}
\|n^\alpha u^\alpha-n^\infty u^\infty\|_{L^\infty_tL^2_x}&\lesssim\|n^\alpha-n^\infty\|_{L^\infty_tL^2_x}\|u^\alpha\|_{L^\infty_tH^2_x}\\
&\phantom{\quad\,}+\|n^\infty(u_0^\alpha-u_0^\infty)\|_{L^\infty_tL^2_x}+\bigg\|n^\infty\int_0^t(v^\alpha-v^\infty)\bigg\|_{L^\infty_tL^2_x}\\
&\lesssim\left(\|n_0^\alpha-n_0^\infty\|_{L^2}+T\|\d_tn^\alpha-\d_tn^\infty\|_{L^\infty_tL^2_x}\right)M\\
&\phantom{\quad\,}+\|n^\infty\|_{L^\infty_tL^2_x}\|u_0^\alpha-u_0^\infty\|_{H^2_x}\\
&\phantom{\quad\,}+T^{1/2}\|n^\infty\|_{L^\infty_tH^1_x}\|v^\alpha-v^\infty\|_{L^2_tL^6_x}.
\end{aligned}
\end{equation*}
By summing up the two contribution in \eqref{eq:u_diff} we then get
\begin{equation*}
\begin{aligned}
\|\Q\Delta(u^\alpha-u^\infty)\|_{L^\infty_tL^2_x}\lesssim &\|u_0^\alpha-u_0^\infty\|_{H^2}+\|n_0^\alpha-n_0^\infty\|_{H^1}\\
&+T^{1/2}\left(\|u^\alpha-u^\infty\|_{L^\infty_tH^2_x}+\|v^\alpha-v^\infty\|_{L^2_tL^6_x}+\|n^\alpha-n^\infty\|_{\mathcal W^1}\right).
\end{aligned}
\end{equation*}
Finally, we notice that, by using the Schr\"odinger equations in \eqref{ZG} and \eqref{ZL}, we have
\begin{equation*}
\|u^\alpha-u^\infty\|_{L^\infty_tL^2_x}\lesssim T\left(\|n^\alpha-n^\infty\|_{L^\infty_tH^1_x}+\|u^\alpha-u^\infty\|_{L^\infty_tH^2_x}\right),
\end{equation*}
so that
\begin{equation*}
\begin{aligned}
\|u^\alpha-u^\infty\|_{L^\infty_tH^2_x}&\lesssim\|u_0^\alpha-u_0^\infty\|_{H^2}+\|n_0^\alpha-n_0^\infty\|_{H^1}+\|\P u^\a_0\|_{H^2}+\a^{-1}\\
&\phantom{\quad\,}+T^{1/2}\left(\|u^\alpha-u^\infty\|_{L^\infty_tH^2_x}+\|v^\alpha-v^\infty\|_{L^2_tL^6_x}+\|n^\alpha-n^\infty\|_{\mathcal W^1} \right).
\end{aligned}
\end{equation*}
Now we put everything together, we finally obtain
\begin{equation*}
\begin{aligned}
\|v^\alpha-v^\infty&\|_{L^2_tL^6_x}+\|n^\alpha-n^\infty\|_{\mathcal W^1}+\|u^\alpha-u^\infty\|_{L^\infty_tH^2_x}\lesssim\\
&\lesssim\|\P u_0^\alpha\|_{H^2}+\alpha^{-1}+\|u_0^\alpha-u_0^\infty\|_{H^2}+\|n_0^\alpha-n_0^\infty\|_{H^1}+\|n_1^\alpha-n_1^\infty\|_{L^2}\\
&\phantom{\quad\,}+T^{1/2}\left(\|u^\alpha-u^\infty\|_{L^\infty_tH^2_x}+\|v^\alpha-v^\infty\|_{L^2_tL^6_x}+\|n^\alpha-n^\infty\|_{\mathcal W^1}\right).
\end{aligned}
\end{equation*}
By choosing $T$ small enough depending on $M$ we can infer
\begin{equation*}
\begin{aligned}
\|v^\alpha-v^\infty&\|_{L^2_tL^6_x}+\|n^\alpha-n^\infty\|_{\mathcal W^1}+\|u^\alpha-u^\infty\|_{L^\infty_tH^2_x}\lesssim\\
&\lesssim\|\P u_0^\alpha\|_{H^2}+\alpha^{-1}+\|u_0^\alpha-u_0^\infty\|_{H^2}+\|n_0^\alpha-n_0^\infty\|_{H^1}+\|n_1^\alpha-n_1^\infty\|_{L^2}.
\end{aligned}
\end{equation*}
This proves the convergence in the time interval $[0, T]$, for $T>0$ small enough. Let now $0<T_1$ be as in the statement of Lemma, we can divide $[0, T_1]$ into many subintervals of length $T$ such that the convergence holds in any small interval. By gluing them together we prove the Lemma.

\end{proof}


\bibliographystyle{amsalpha}

\end{document}